\documentclass[12pt]{article}
\usepackage{amsmath}
\usepackage{amssymb}
\usepackage{amsthm}
\usepackage{float}
\usepackage{url}
\usepackage{tikz,pgf}
\usepackage{tikz-3dplot}
\usetikzlibrary{math}
\usetikzlibrary{arrows.meta}

\title{An Elliptic Curve Governing Hopf Linking in an $A_4$-Symmetric Tensegrity}
\author{Taizo Sadahiro}
\date{}

\newtheorem{prop}{Proposition}
\newtheorem{thm}{Theorem}
\newtheorem{rmk}{Remark}

\usepackage{graphicx}	
\begin{document}

\maketitle

\begin{abstract}
We study in detail an $A_4$-symmetric tensegrity appearing in Connelly's catalog.
The realizable configurations form a one-parameter family that can be
parametrized by points on the elliptic curve with Cremona label 30a2.
The curve has only twelve rational points, 
among which only one corresponds to a stable tensegrity configuration 
whose cable framework forms a cuboctahedron.
From a topological viewpoint, however, the underlying pair of the
strut triangles
preserves a Hopf link structure throughout the entire interval $0<\omega_1<1$
of the stress parameter.
\end{abstract}

\section{Introduction}

Tensegrity structures provide a rich interaction between
geometry, combinatorics, and rigidity theory.
In his catalog  \cite{ConnellyCatalog, ConnellySite, connelly2022frameworks} of symmetric tensegrities,
Connelly described many remarkable examples
whose configurations depend on continuous stress parameters.

In this paper we study in detail a tensegrity with $A_4$ symmetry
appearing in Connelly's catalog.
Combinatorially, the underlying graph is the Cayley graph
\[
{\rm Cay}\bigl(A_4,\{(1,2,3),(1,3,4),(2,4,3)\}\bigr),
\]
in which the edges generated by $(1,2,3)$ form four disjoint triangles.
These triangles play the role of the struts of the tensegrity,
while the remaining edges act as cables.

The equilibrium condition leads to a stress matrix depending on two
parameters.
The condition that the matrix has a nontrivial kernel
defines a plane cubic curve,
which turns out to be the elliptic curve with the Cremona label 30a2
in the LMFDB database \cite{LMFDB}.
Each point on this curve determines a realization of the graph
as a tensegrity in $\mathbb{R}^3$.

\begin{figure}[H]
\begin{center}
 \includegraphics[bb=0 300 1536 1748,clip,width=6cm]{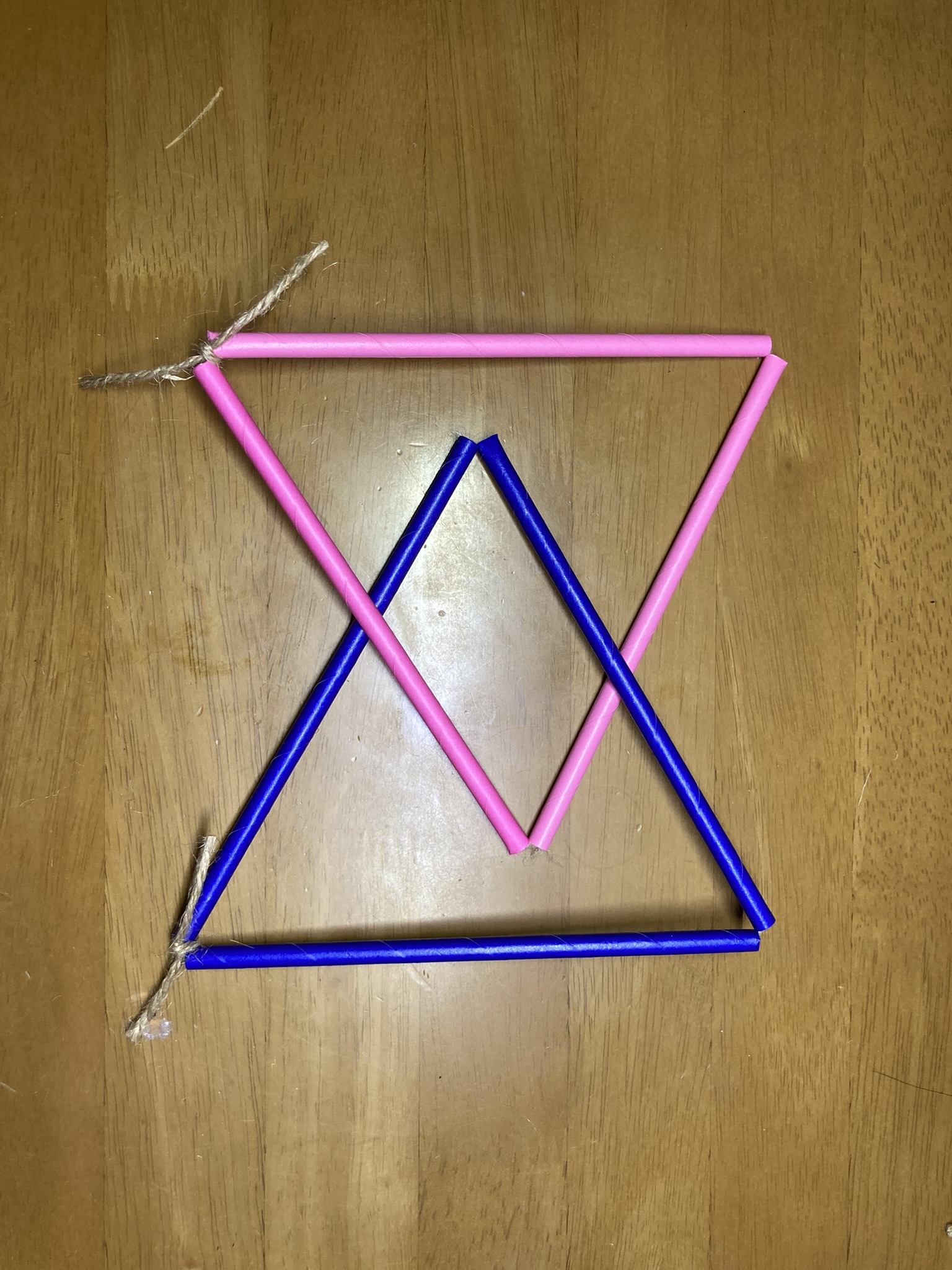}
 \includegraphics[bb=0 300 1536 1748,clip,width=6cm]{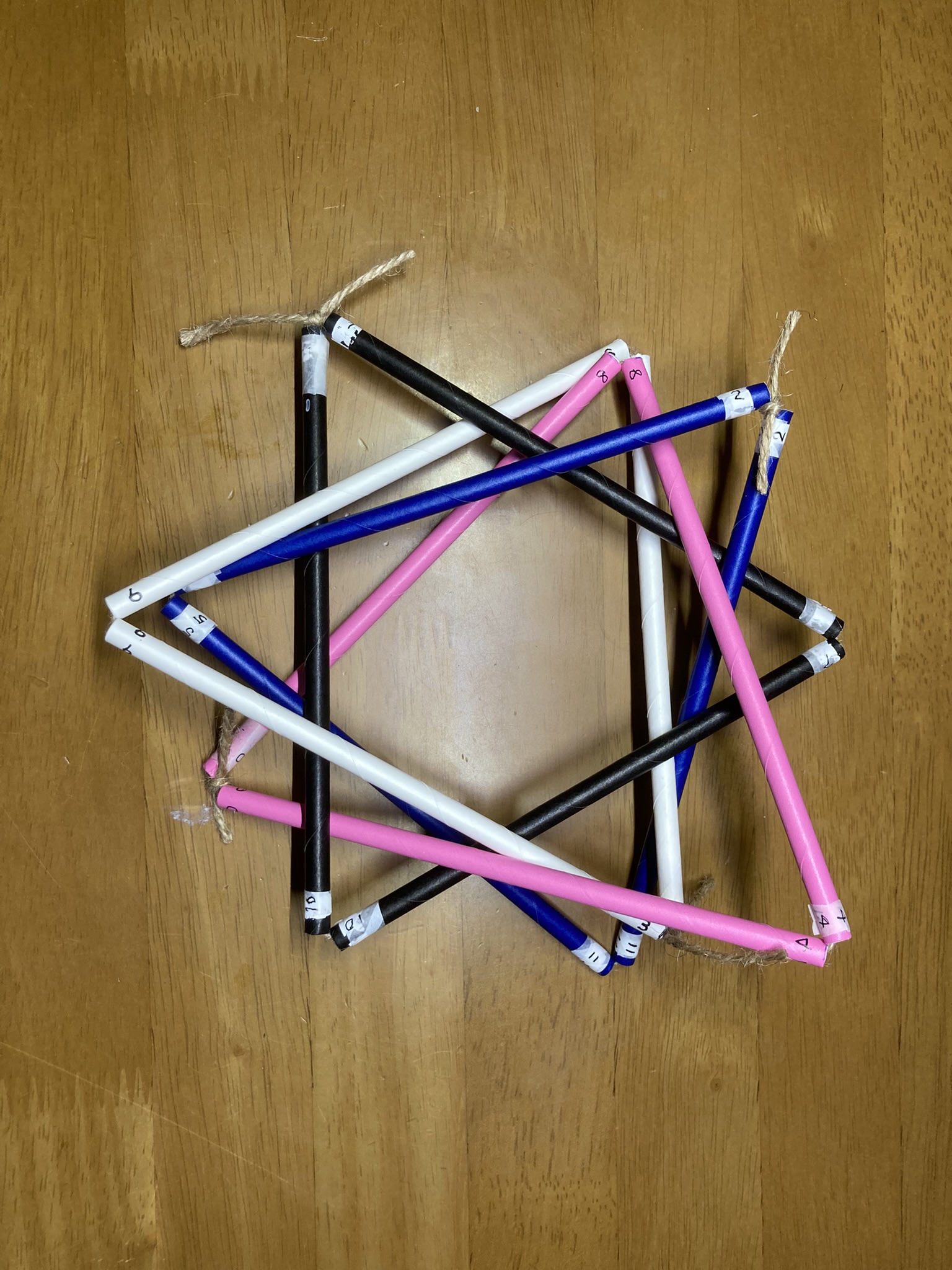}
\end{center}
 \caption{A Hopf link formed of straw triangles (left).
 A 4-component link in which each pair of triangles is Hopf linked (right).}
 \label{fig:links}
\end{figure}

\begin{figure}[H]
 \begin{center}
  \includegraphics[bb=350 500 1186 1348,clip,width=5cm]{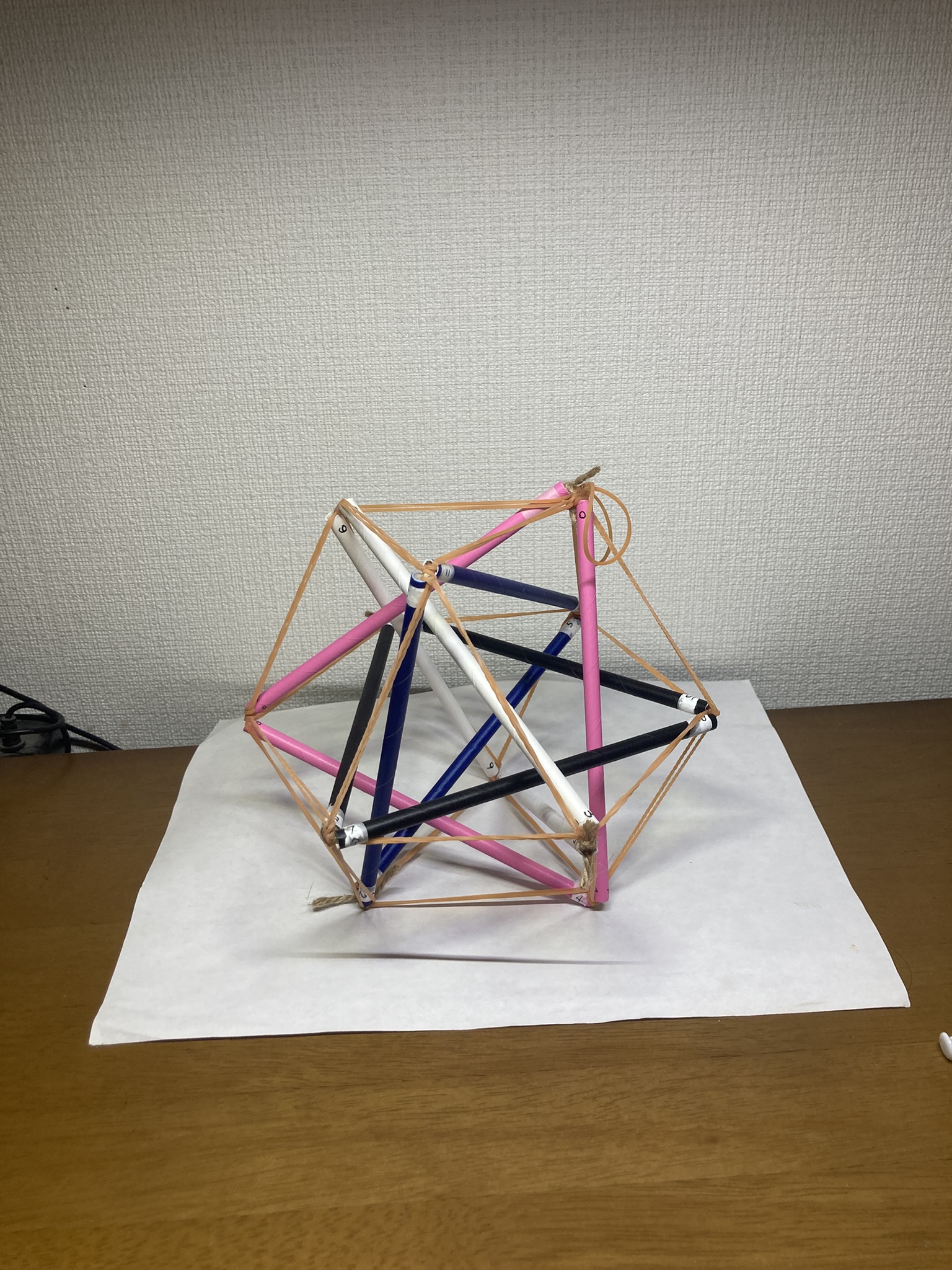}
  \includegraphics[bb=0 0 852 777,clip,width=5cm]{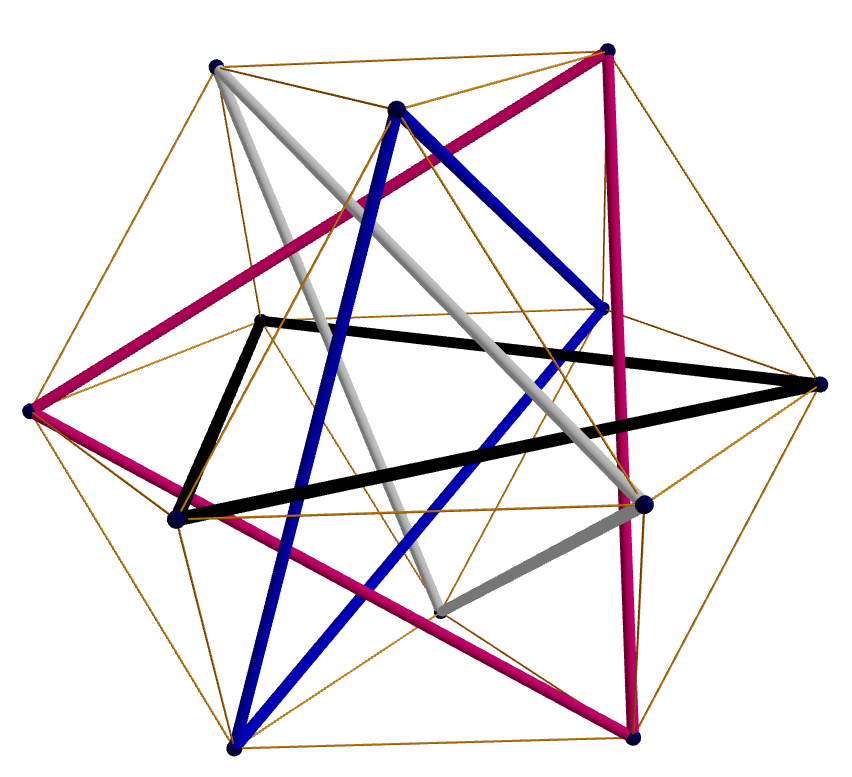}
  \includegraphics[bb=0 0 852 777,clip,width=5cm]{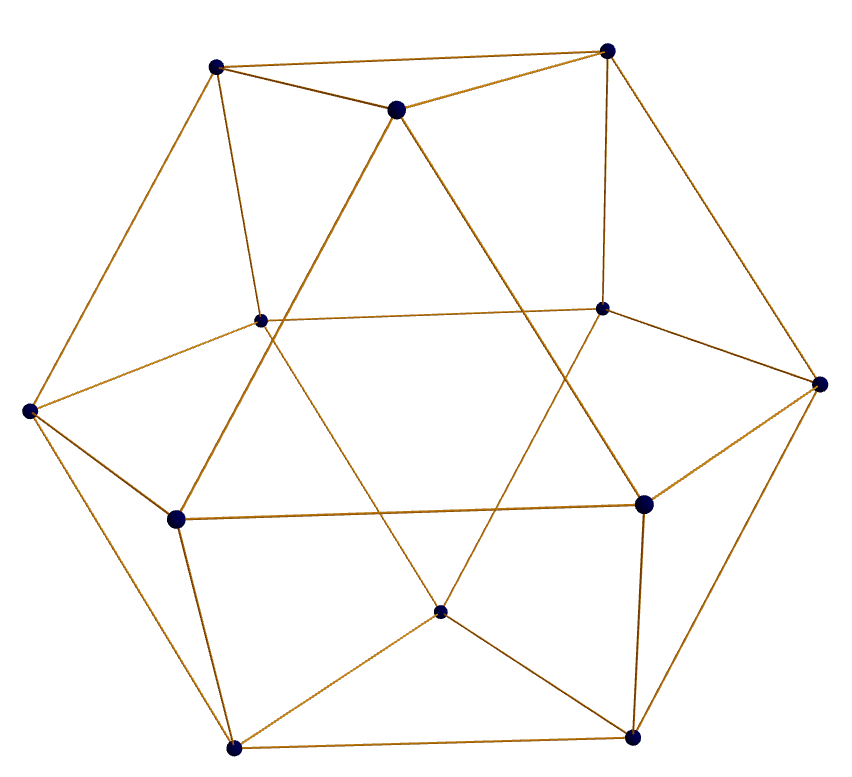}
  \caption{By adding $24$ cables (rubber bands) to the 4-component link in Figure $\ref{fig:links}$,
  An $A_4$-symmetric tensegrity is physically realized (left).
  Its computer generated image (middle).
  The cables form the  skelton graph of the cuboctahedron (right).}
  \label{fig:physicalrealization}
 \end{center}
\end{figure}

In Connelly's catalog, the equilibrium stresses are written as
$(\omega_1,\omega_{-1})$, where $\omega_1$ is the stress on the cables
and $\omega_{-1}$ is the stress on the struts.
For readability, we write these parameters as $(x,y)$ in the present paper.
Thus $x=\omega_1$ and $y=\omega_{-1}$.
\begin{thm}
For every parameter value satisfying $0<x<1$,
the four triangular strut components form
a $4$-component mutual Hopf link.
\end{thm}

The proof proceeds as follows.
First we analyze a distinguished rational point
$
(x,y)=\left(\frac12,-\frac13\right),
$
which corresponds to the cuboctahedral configuration.
The physical realization of the cuboctahedral configuration
is illustrated in Figure $\ref{fig:links}$ and $\ref{fig:physicalrealization}$.
At this point the Hopf linking can be verified directly.
We then express the intersection parameters between the triangles
as rational functions on the elliptic curve.
Using resultant computations we show that these functions have
no zeros or poles in the interval $0<x<1$,
which implies that the linking structure cannot change.

Thus the Hopf link structure persists for all admissible parameter
values.
The proof combines elementary geometry with explicit algebraic
computations carried out using a computer algebra system.

\section{Construction of the tensegrity}

We begin by explaining how this tensegrity is constructed
in \cite{connelly2022frameworks}.
Viewed as a graph, it is isomorphic to the Cayley graph
\[
 {\rm Cay}\bigl(A_4,\{s=(1,2,3),\, c_1=(1,3,4),\, c_2=(2,4,3)\}\bigr).
\]
A {\em strut} corresponds to an edge generated by $s=(1,2,3)$,
and a {\em cable} corresponds to an edge generated by $(1,3,4)$ or $(2,4,3)$.
Therefore, the struts form four pairwise disjoint triangles.

We use the standard three-dimensional irreducible representation
$\rho\colon A_4\to {\rm GL}(3,\mathbb{Z})$.
\[
 \rho(g_1) =
 \begin{pmatrix}
  0 & 1 & 0\\
  0 & 0 & -1\\
  -1 & 0 & 0
 \end{pmatrix},
 \qquad
 \rho(g_2) =
 \begin{pmatrix}
  1 & 0 & 0\\
  0 & -1 & 0\\
  0 & 0 & -1
 \end{pmatrix},
 \qquad
 \rho(g_3) =
 \begin{pmatrix}
  -1 & 0 & 0\\
  0 & 1 & 0\\
  0 & 0 & -1
 \end{pmatrix},
\]
where
\[
 g_1=s=(1,2,3), \qquad g_2=(1,2)(3,4), \qquad g_3=(1,3)(2,4).
\]

Then $\rho(s)=\rho(g_1)$, and
\[
 \rho(c_1)=\rho(g_1g_2)
 =
 \begin{pmatrix}
  0 & -1 & 0\\
  0 & 0 & 1\\
  -1 & 0 & 0
 \end{pmatrix},
 \qquad
 \rho(c_2)=\rho(g_1g_3)
 =
 \begin{pmatrix}
  0 & 1 & 0\\
  0 & 0 & 1\\
  1 & 0 & 0
 \end{pmatrix}.
\]
Thus we obtain the stress matrix
\begin{align*}
 \Omega(x,y)
 &= x\bigl(\rho(c_1)+\rho(c_1^{-1})\bigr)
  + (1-x)\bigl(\rho(c_2)+\rho(c_2^{-1})\bigr)
  + y\bigl(\rho(s)+\rho(s^{-1})\bigr)
  - 2(1+y)I_3 \\
 &=
 \begin{pmatrix}
  -2y-2 & -2x+y+1 & -2x-y+1 \\
  -2x+y+1 & -2y-2 & -y+1 \\
  -2x-y+1 & -y+1 & -2y-2
 \end{pmatrix},
\end{align*}
where $x$ (resp. $1-x$) 
is the stress parmeter 
of the cable corresponding to the generator $(1,3,4)$
(resp. $(2,4,3)$),
and $y$ is the stress parameter 
of the strut.

To realize the Cayley graph as a tensegrity in $\mathbb{R}^3$,
the equation
\begin{equation}
 \label{eq:null}
 \Omega(x,y)\mathbf{p}_0=\mathbf{0}
\end{equation}
must have a nontrivial solution in $\mathbb{R}^3$.
Hence we must have $\det\Omega(x,y)=0$.
A direct computation gives
\begin{equation}
 \label{eq:detOmega-affine}
 \det \Omega(x,y)
 = 8\bigl(x^2y + 3x^2 - xy - 3y^2 - 3x - 3y\bigr)=0.
\end{equation}
It is shown in \cite{connelly2022frameworks} that the parameter $(x,y)$
gives a stable tensegrity whenever $(x,y)$ lies on the arc of the curve
defined by \eqref{eq:detOmega-affine} with $0<x<1$ and $-\frac{1}{3}\le y<0$.
See Figure $\ref{fig:stresspars}$.
\begin{center}
 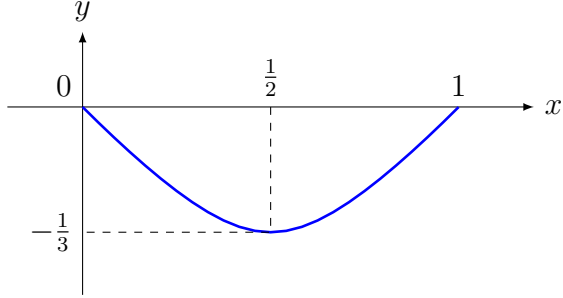
\begin{figure}[H]
  \begin{center}
  \begin{tikzpicture}[scale=5,
    declare function={
      g(\x) = 1/6*\x*\x - 1/6*\x + 1/6*sqrt(\x^4 - 2*\x^3 + 31*\x^2 - 30*\x + 9) - 1/2;
   }
   ]
   \draw[->, >=latex] (-0.2,0) -- (1.2,0) node[right] {$x$};
   \draw[->, >=latex] (0,-0.5) -- (0,0.2) node[above] {$y$};
   \draw [blue, domain=0:1, line width = 1] plot (\x, {g(\x)});
   \draw (0,0) node[above left] {$0$};
   \draw (1,0) node[above] {$1$};
   \draw ({1/2},0) node[above]{$\frac{1}{2}$};
   \draw (0,{-1/3}) node[left]{$-\frac{1}{3}$};
   \draw [dashed] ({1/2},0) -- ({1/2},{-1/3}) -- (0,{-1/3});
  \end{tikzpicture}
   \caption{Stress parameters $(x,y)$ which give stable tensegrities.}
   \label{fig:stresspars}
  \end{center}
 \end{figure}
\end{center}

A direct computation shows that
\[
\mathbf{p}_0=
\begin{pmatrix}
 -2xy + 3y^2 - 6x + 2y + 3\\
 -4x^2 - 4xy + 3y^2 + 4x + 10y + 3\\
 4x^2 - 3y^2 - 4x + 3
\end{pmatrix}
\]
satisfies \eqref{eq:null} under the condition \eqref{eq:detOmega-affine}.

For such a vector $\mathbf{p}_0$, the set of nodes (or vertices) of the tensegrity
is obtained as the $A_4$-orbit of $\mathbf{p}_0$, namely
\[
 V=\{\rho(g)\mathbf{p}_0 \mid g\in A_4\}.
\]
Two nodes $\rho(g)\mathbf{p}_0$ and $\rho(h)\mathbf{p}_0$
are connected by a strut if $h=gs$,
and by a cable if $h=gc_1$ or $h=gc_2$.
Figure $\ref{fig:deform}$ shows the
deformation of the tensegrity as $x$ varies from $0$ to $1$.

\begin{center}
 \begin{figure}[H]
  \begin{tikzpicture}[scale=0.9]
   \draw (0,0) node{\includegraphics[bb=0 0 814 696,clip,width=3.5cm]{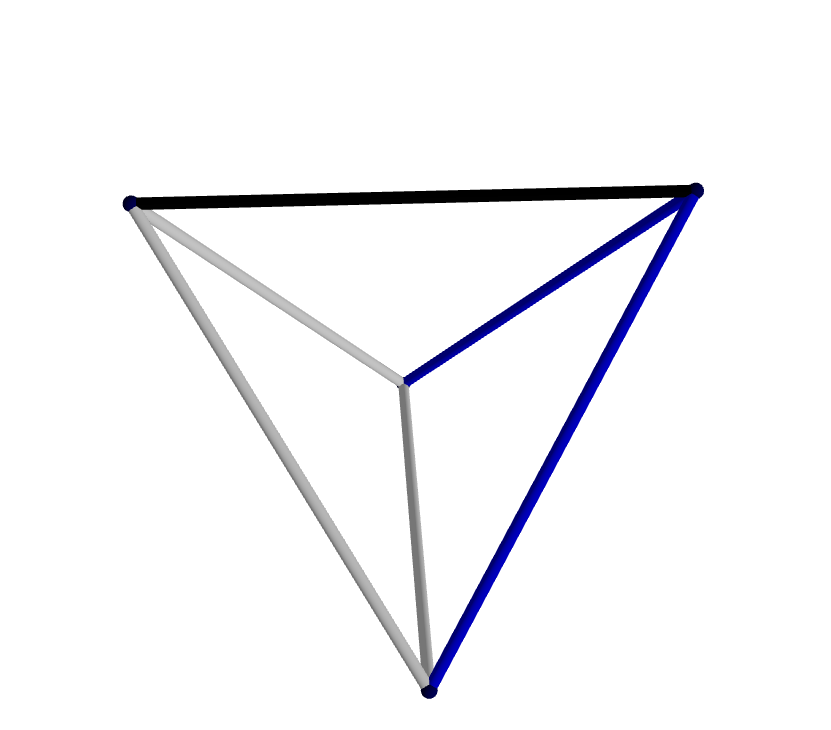}};
   \draw (0,1.5) node {$x=0$};
   \draw (3.5,0) node{\includegraphics[bb=0 0 814 696,clip,width=3.5cm]{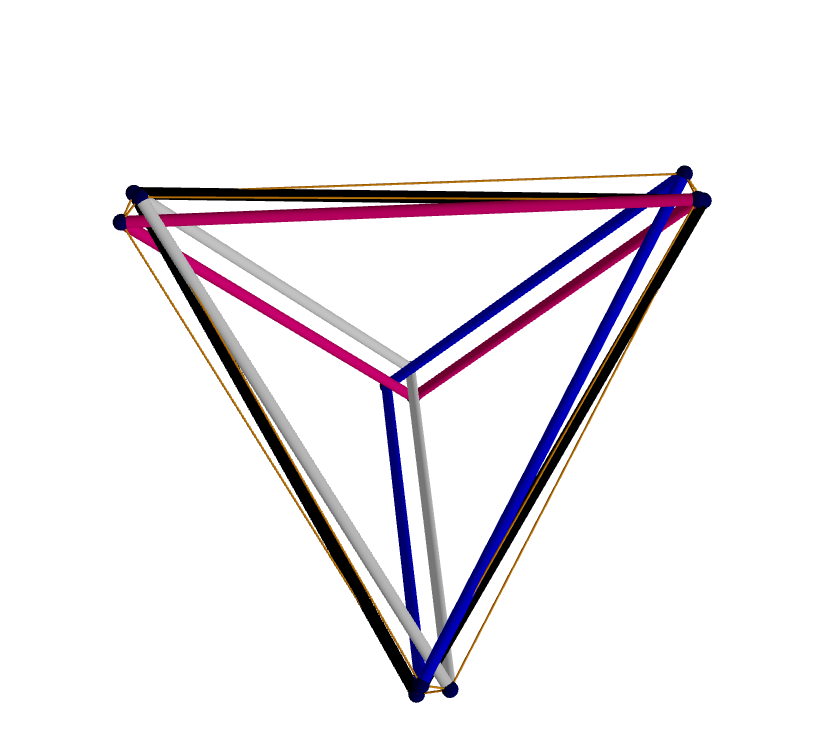}};
   \draw (3.5,1.5) node {$x=0.1$};
   \draw (7,0) node{\includegraphics[bb=0 0 814 696,clip,width=3.5cm]{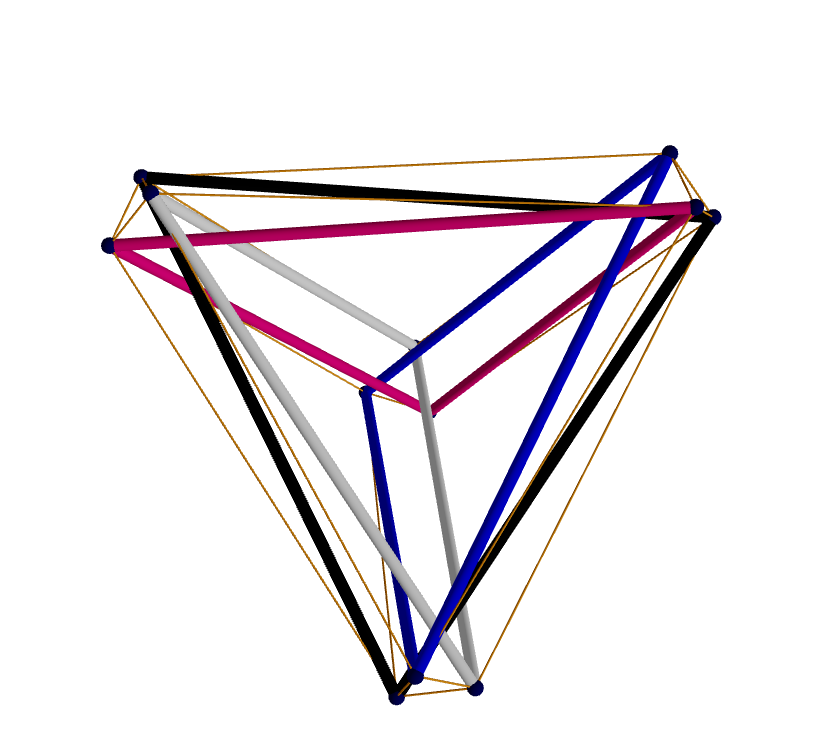}};
   \draw (7,1.5) node {$x=0.2$};
   \draw (10.5,0) node{\includegraphics[bb=0 0 814 696,clip,width=3.5cm]{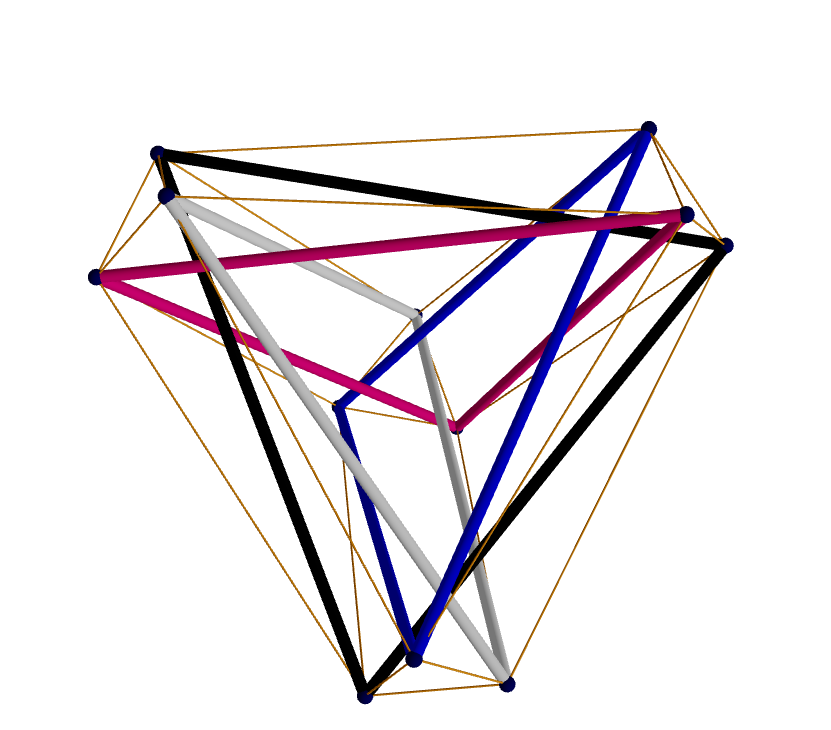}};
   \draw (10.5,1.5) node {$x=0.3$};
   \draw (14,0) node{\includegraphics[bb=0 0 814 696,clip,width=3.5cm]{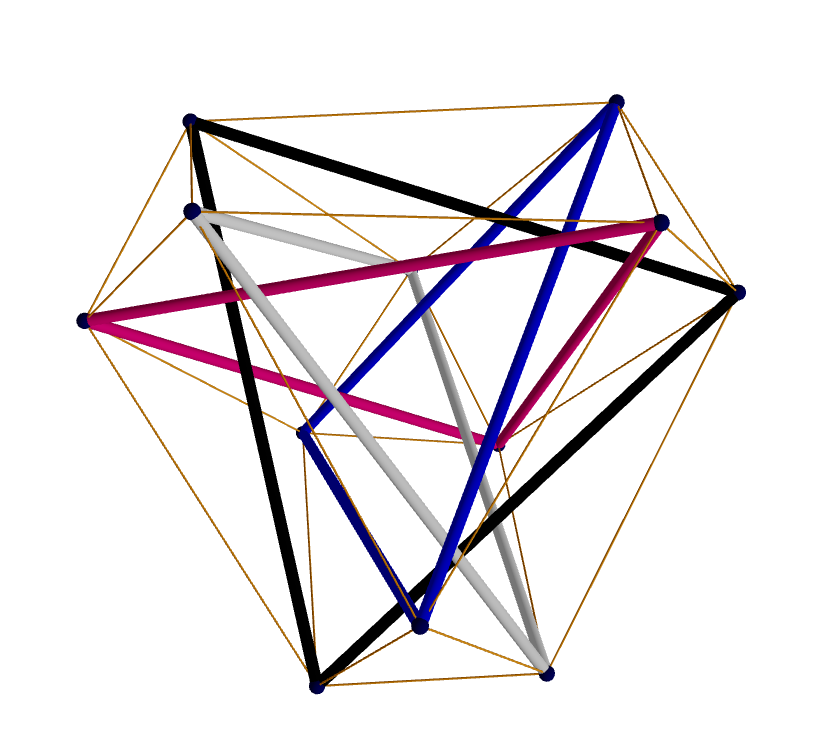}};
   \draw (14,1.5) node {$x=0.4$};
   \draw (14,-3.5) node{\includegraphics[bb=0 0 814 696,clip,width=3.5cm]{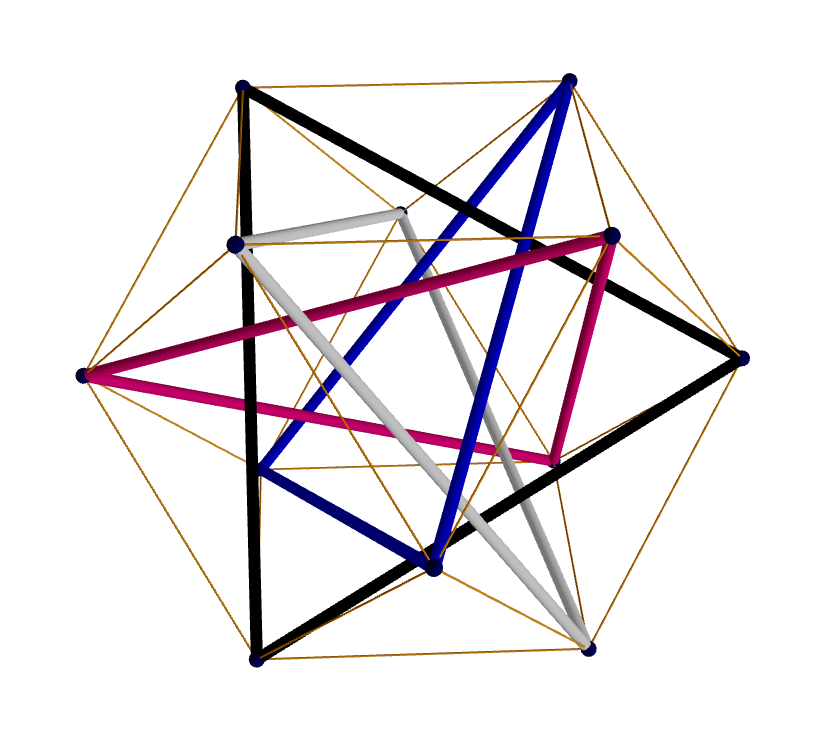}};
   \draw (11.5,-3.5) node {$x=0.5$};
   \draw (14,-7) node{\includegraphics[bb=0 0 814 696,clip,width=3.5cm]{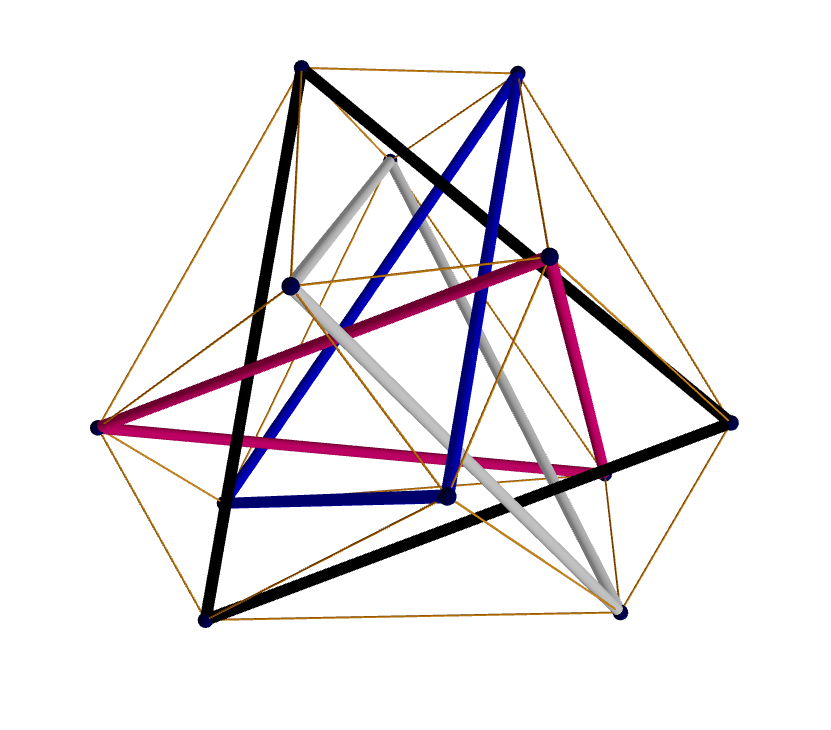}};
   \draw (14,-8.5) node {$x=0.6$};
   \draw (10.5,-7) node{\includegraphics[bb=0 0 814 696,clip,width=3.5cm]{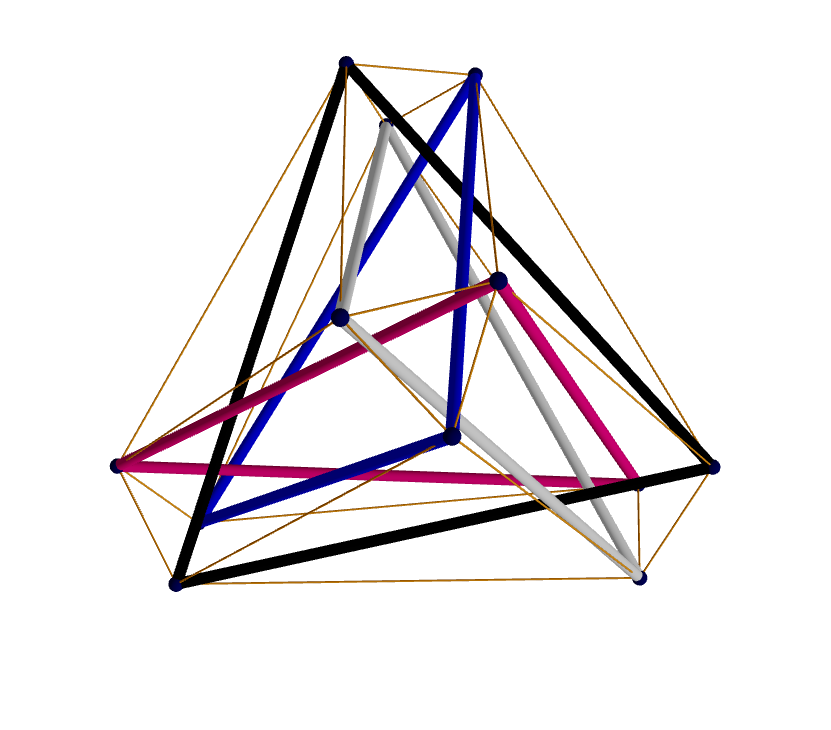}};
   \draw (10,-8.5) node {$x=0.7$};
   \draw (7,-7) node{\includegraphics[bb=0 0 814 696,clip,width=3.5cm]{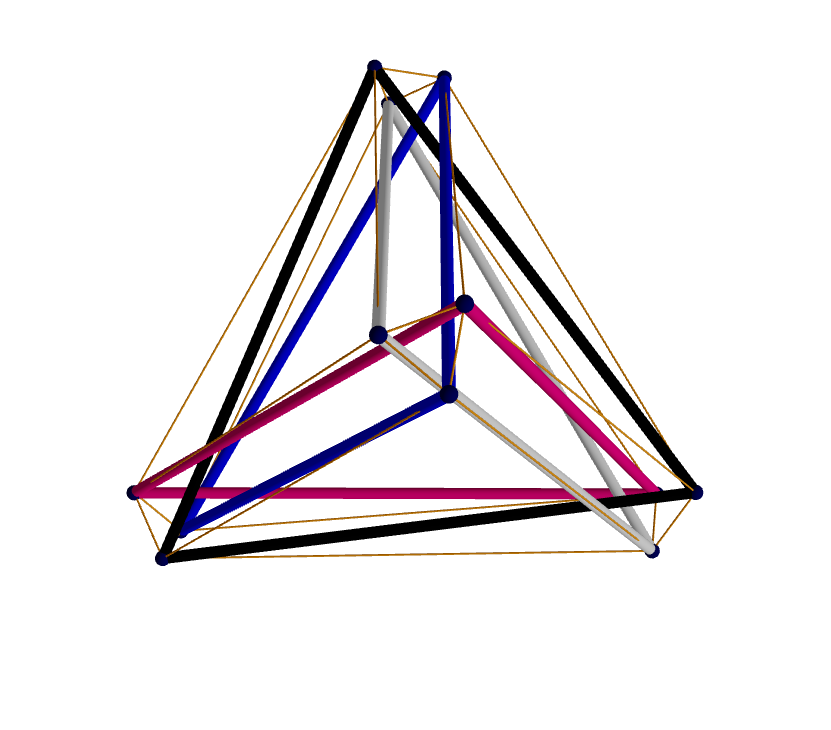}};
   \draw (7,-8.5) node {$x=0.8$};
   \draw (3.5,-7) node{\includegraphics[bb=0 0 814 696,clip,width=3.5cm]{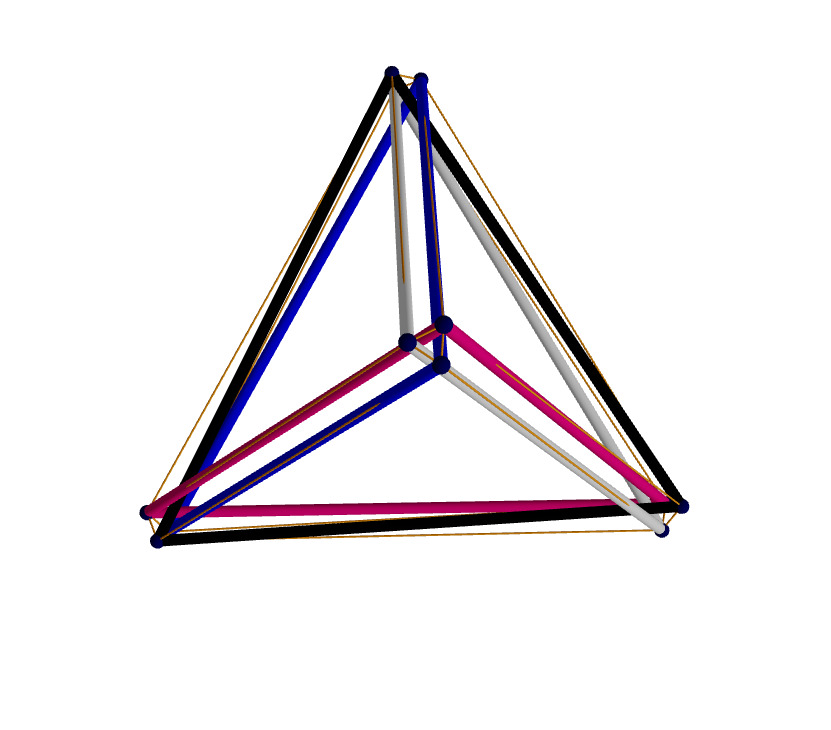}};
   \draw (3.5,-8.5) node {$x=0.9$};
   \draw (0,-7) node{\includegraphics[bb=0 0 814 696,clip,width=3.5cm]{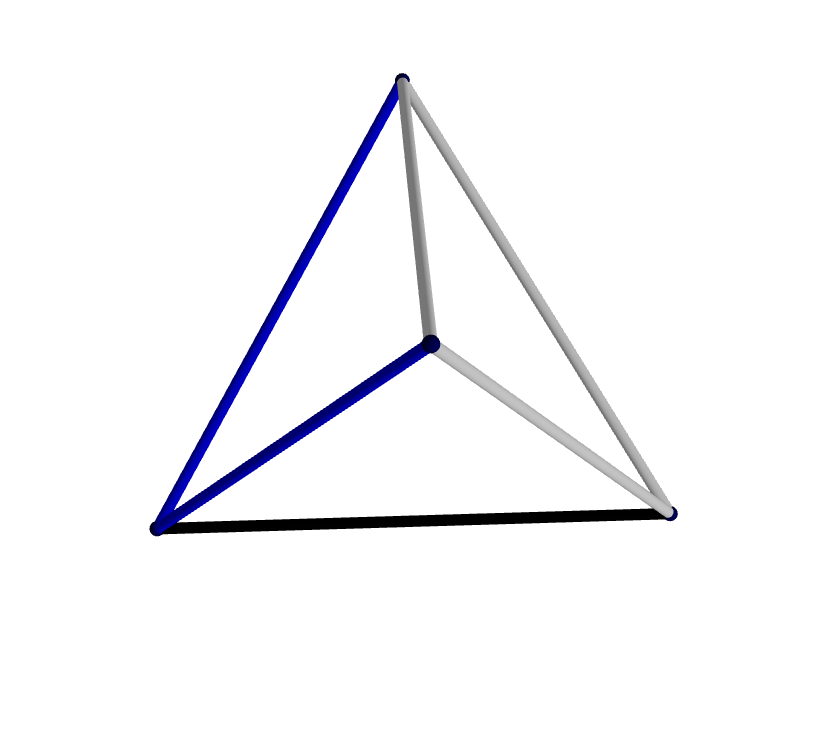}};
   \draw (0,-8.5) node {$x=1$};
  \end{tikzpicture}
  \caption{Deformation of the tensegrity as $x$ varies from $0$ to $1$.}
  \label{fig:deform}

 \end{figure}
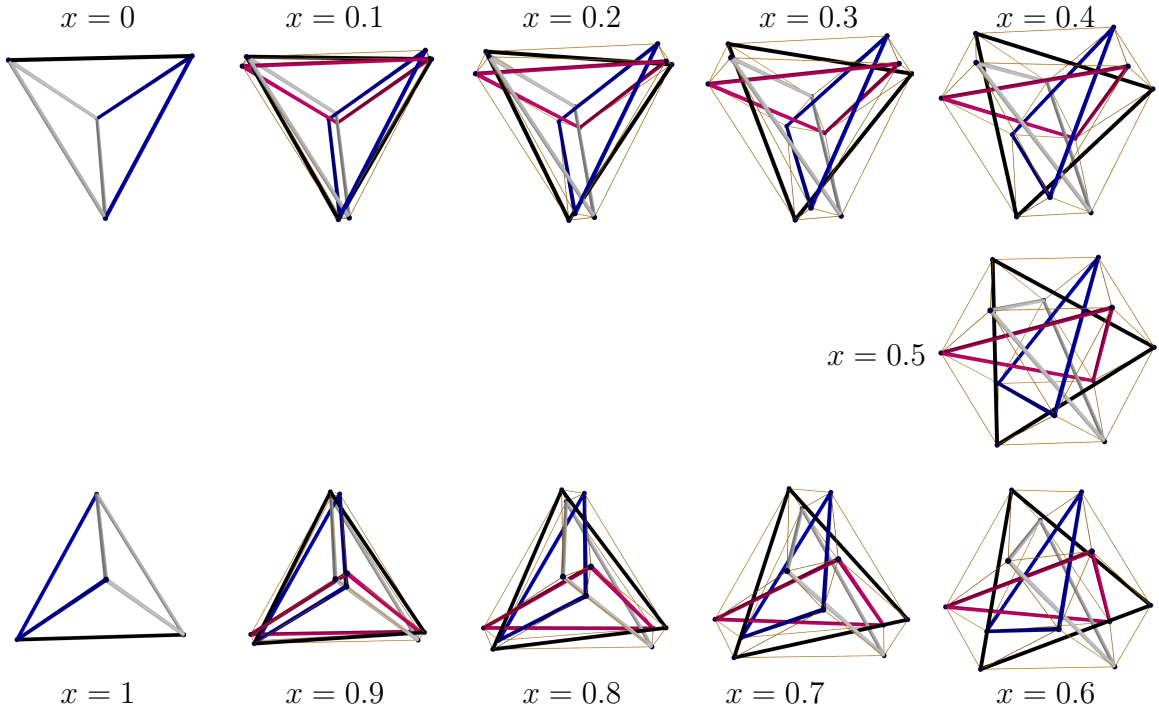
\end{center}

\begin{figure}[H]
\begin{center}
\tdplotsetmaincoords{60}{120}
\begin{tikzpicture}[scale=2, tdplot_main_coords]
\draw[ultra thick, blue] (1, 0, -1) node[right]{$\rho(cs^2){\mathbf p}_0$} -- (0,-1,1) node[left]{$\rho(cs){\mathbf p}_0$} 
 -- (-1,1,0) node[right]{$\rho(c){\mathbf p}_0$} -- cycle;  
\draw[ultra thick, red] (0,1,1)node[right]{${\mathbf p}_0$} -- (1,-1,0) node[left]{$\rho(s){\mathbf p}_0$}-- (-1,0,-1) -- (0,1,1);
\fill[opacity=1,white] (0,1,1) -- (1,-1,0) -- (-1,0,-1) -- (0,1,1);

\draw[ultra thick, red] (0, 1, 1) -- (-1,0,-1) node[right]{$\rho(s^2){\mathbf p}_0$};  
\draw[ultra thick, red] (1, -1, 0) -- (-1,0,-1);  
\draw[ultra thick, blue] (-1, 1, 0) -- (0,-1,1);
\draw[ultra thick, blue] (0,-1,1) -- ({1/2}, {-1/2},0);
\fill[opacity=0.1,blue] (0,1,1) -- (1,-1,0) -- (-1,0,-1) -- (0,1,1);
\end{tikzpicture}
 \caption{$\Delta$ and $\Delta'$ for $(x,y)=\left(\frac{1}{2}, -\frac{1}{3}\right)$.}
 \label{fig:deltadelta}
\end{center}
\end{figure}
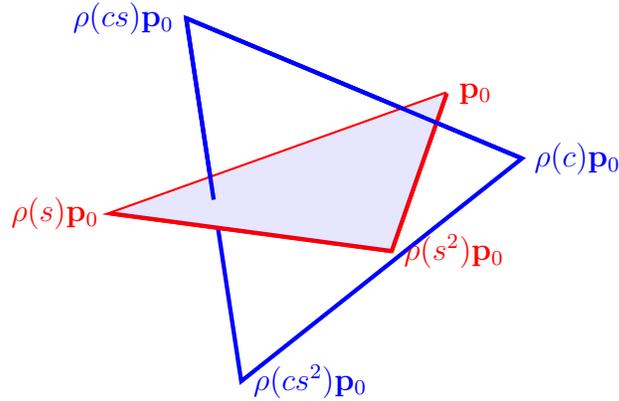

Then, we consider the link structure
of the four triangles composed of struts.
More precisely, we show that
every pair of the two triangles forms a Hopf link.
Let $\Delta$ be the triangle composed of
the struts connecting the nodes
$
{\mathbf p}_0, ~ \rho(s){\mathbf p}_0, \rho(s^2){\mathbf p}_0,$
and
let $\Delta'$ be the triangle composed of
the struts connecting the nodes
$\rho(c_1){\mathbf p}_0, ~ \rho(c_1s){\mathbf p}_0, \rho(c_1s^2){\mathbf p}_0.$
Figure $\ref{fig:deltadelta}$ illustrates the relative position
of $\Delta$ and $\Delta'$ for $(x,y)=\left(\frac{1}{2}, -\frac{1}{3}\right)$,
which is rigorously stated and proved as follows.
%

\begin{prop}
\label{prop:onehopflink}
At the distinguished point
$
(x,y)=\left(\frac12,-\frac13\right),
$
the two triangular strut components
\[
\Delta=\{{\mathbf p}_0,\rho(s){\mathbf p}_0,\rho(s^2){\mathbf p}_0\}
\quad\text{and}\quad
\Delta'=\{\rho(c_1){\mathbf p}_0,\rho(c_1s){\mathbf p}_0,\rho(c_1s^2){\mathbf p}_0\}
\]
form a Hopf link.
\end{prop}
\begin{proof}
At the distinguished point \((x,y)=\left(\frac12,-\frac13\right)\), the general expression for the null vector gives
\[
{\mathbf p}_0=\left(0,\frac53,\frac53\right).
\]
Since the realization is unchanged up to an overall scalar multiple, we normalize and write
\[
{\mathbf p}_0=(0,1,1).
\]
Let
\[
A={\mathbf p}_0=
\begin{pmatrix}
0\\ 1\\ 1
\end{pmatrix},\qquad
B=\rho(s){\mathbf p}_0=
\begin{pmatrix}
1\\ -1\\ 0
\end{pmatrix},\qquad
C=\rho(s^2){\mathbf p}_0=
\begin{pmatrix}
-1\\ 0\\ -1
\end{pmatrix},
\]
so that
\[
\Delta=\triangle ABC.
\]
Similarly, let
\[
D=\rho(c_1){\mathbf p}_0=
\begin{pmatrix}
-1\\ 1\\ 0
\end{pmatrix},\qquad
E=\rho(c_1s){\mathbf p}_0=
\begin{pmatrix}
1\\ 0\\ -1
\end{pmatrix},\qquad
F=\rho(c_1s^2){\mathbf p}_0=
\begin{pmatrix}
0\\ -1\\ 1
\end{pmatrix},
\]
so that
\[
\Delta'=\triangle DEF.
\]

The plane containing \(\Delta\) is
\[
\Pi:\ X+Y-Z=0.
\]
Indeed, the three vertices \(A,B,C\) satisfy
\[
0+1-1=0,\qquad 1-1-0=0,\qquad -1+0-(-1)=0.
\]

We claim that the edge \(EF\) meets the interior of the triangular disk bounded by \(\Delta\).
Its midpoint is
\[
\frac12E+\frac12F
=
\frac12
\begin{pmatrix}
1\\ -1\\ 0
\end{pmatrix}
=
\begin{pmatrix}
\frac12\\ -\frac12\\ 0
\end{pmatrix},
\]
and this point lies on \(\Pi\), since
\[
\frac12-\frac12-0=0.
\]
Moreover,
\[
\frac12E+\frac12F
=
B+\frac16(A-B)+\frac16(C-B).
\]
Since
\[
\frac16>0,\qquad
1-\frac16-\frac16=\frac23>0,
\]
this point is a strict convex combination of \(A,B,C\). Hence it lies in the interior of the triangular disk bounded by \(\Delta\). Therefore the edge \(EF\) intersects the interior of that disk.
In particular, the intersection point lies on the symmetry axis of the equilateral triangle 
$\Delta$ through the vertex B.

Next we show that the other two edges of \(\Delta'\) do not meet the triangular disk bounded by \(\Delta\).
First,
\[
D = B + \frac23(A-B) +\frac23(C-B),
\]
so \(D\in \Pi\), but \(D\) lies outside \(\Delta\), because one barycentric coefficient is negative.

Now consider the function
\[
\ell(X,Y,Z)=X+Y-Z.
\]
Since \(\Pi\) is given by \(\ell=0\), we have
\[
\ell(D)= -1+1-0=0,\qquad
\ell(E)=1+0-(-1)=2,\qquad
\ell(F)=0+(-1)-1=-2.
\]
Thus the segment \(DE\) meets the plane \(\Pi\) only at the endpoint \(D\), and likewise the segment \(DF\) meets \(\Pi\) only at the endpoint \(D\). Since \(D\) lies outside \(\Delta\), neither \(DE\) nor \(DF\) intersects the triangular disk bounded by \(\Delta\).

Consequently, the closed polygon \(\Delta'\) meets the triangular disk bounded by \(\Delta\) in exactly one point, namely the interior point of the edge \(EF\) described above. Therefore
\[
\bigl|\operatorname{lk}(\Delta,\Delta')\bigr|=1.
\]
Since both \(\Delta\) and \(\Delta'\) are embedded triangles, each is an unknot. Hence \(\Delta\) and \(\Delta'\) form a Hopf link.
\end{proof}

\begin{thm}
At the distinguished point
$
(x,y)=\left(\frac12,-\frac13\right),
$
the four triangular strut components are pairwise Hopf linked.
\end{thm}

\begin{proof}
By Proposition \ref{prop:onehopflink},
the representative pair \(\Delta\) and \(\Delta'\) forms a Hopf link.
Since the action of \(A_4\) on the four triangular components is transitive on unordered pairs, every pair of distinct components is equivalent to this one. Hence every two distinct triangular components form a Hopf link.
\end{proof}

\section{Persistence of the Hopf link}

In the previous section we proved that,
at the distinguished point
$(x,y)=\left(\frac12,-\frac13\right)$,
the four triangular strut components form a four-component mutual Hopf link.
In this section we show that this link structure persists
for all realizable configurations with $0<x<1$.

\subsection{Intersection parameters}

Let $\Delta$ and $\Delta'$ be the two triangular components
defined in Section~2.
To determine whether the strut in $\Delta'$ connecting
$\rho(c_1s){\mathbf p}_0$ and $\rho(c_1s^2){\mathbf p}_0$
intersects the triangular disk bounded by $\Delta$,
we consider the equation
\[
 \tau\rho(c_1s){\mathbf p}_0
 +
 (1-\tau)\rho(c_1s^2){\mathbf p}_0
 =
 \rho(s){\mathbf p}_0
 +
 R_1({\mathbf p}_0 - \rho(s){\mathbf p}_0 )
 +
 R_2(\rho(s^2){\mathbf p}_0  - \rho(s){\mathbf p}_0).
\]

Solving the intersection equation yields the parameters
\[
\tau=\tau(x,y), \qquad
R_1=R_1(x,y), \qquad
R_2=R_2(x,y),
\]
which are rational functions on the curve \(d(x,y)=0\).

For the parameter \(\tau\) we obtain
\[
\tau(x,y)=\frac{N_\tau(x,y)}{D_\tau(x,y)},
\]
where
\[
N_\tau(x,y)=(-x+y)(2x+3y+1),
\qquad
D_\tau(x,y)=2xy-5y-3 .
\]
Moreover,
\[
R_1(x,y)=\frac{N_1(x,y)}{D(x,y)},\qquad
R_2(x,y)=\frac{N_2(x,y)}{D(x,y)},
\]
where
\[
D(x,y)=(2xy-5y-3)P(x,y),
\]
\[
N_1(x,y)=(-2x+y+3)(x-1)Q_1(x,y),
\]
\[
N_2(x,y)=-(2x^2+xy-5x-4y)Q_2(x,y).
\]

Here
\[
P(x,y)=
4 x^{4} + 3 x^{2} y^{2} - 20 x^{3} - 6 x^{2} y - 15 x y^{2}
 + 31 x^{2} + 6 x y + 21 y^{2} - 15 x + 18 y + 9 .
\]

The explicit forms of the polynomials \(Q_1\) and \(Q_2\)
are recorded in Appendix~A.
Figure ${\ref{fig:tau}}$ shows numerical plots
of $\tau$ and $(R_1, R_2)$ for $0 < x < 1$.

\begin{figure}[H]
  \centering
  \begin{tikzpicture}[scale=3,
    declare function={
      g(\x) = 1/6*\x*\x - 1/6*\x + 1/6*sqrt(\x^4 - 2*\x^3 + 31*\x^2 - 30*\x + 9) - 1/2;
      den(\x,\y) = 8*\x^5*\y + 6*\x^3*\y^3 - 60*\x^4*\y - 12*\x^3*\y^2 - 45*\x^2*\y^3 - 12*\x^4 + 162*\x^3*\y + 33*\x^2*\y^2 + 117*\x*\y^3 + 60*\x^3 - 167*\x^2*\y + 51*\x*\y^2 - 105*\y^3 - 93*\x^2 + 75*\x*\y - 153*\y^2 + 45*\x - 99*\y - 27;
      numR1(\x,\y) = 8*\x^6 + 20*\x^5*\y - 14*\x^4*\y^2 - 17*\x^3*\y^3 + 9*\x^2*\y^4 - 36*\x^5 - 132*\x^4*\y + 39*\x^3*\y^2 + 108*\x^2*\y^3 - 36*\x*\y^4 + 18*\x^4 + 285*\x^3*\y + 66*\x^2*\y^2 - 210*\x*\y^3 + 27*\y^4 + 85*\x^3 - 200*\x^2*\y - 220*\x*\y^2 + 119*\y^3 - 111*\x^2 - 18*\x*\y + 129*\y^2 + 36*\x + 45*\y;
      numR2(\x,\y) = -8*\x^6 + 4*\x^5*\y + 30*\x^4*\y^2 - 5*\x^3*\y^3 - 9*\x^2*\y^4 + 36*\x^5 - 171*\x^3*\y^2 - 15*\x^2*\y^3 + 54*\x*\y^4 - 54*\x^4 - 87*\x^3*\y + 297*\x^2*\y^2 + 135*\x*\y^3 - 72*\y^4 + 23*\x^3 + 239*\x^2*\y - 93*\x*\y^2 - 196*\y^3 + 48*\x^2 - 147*\x*\y - 144*\y^2 - 45*\x - 36*\y;
      R1(\x,\y) = numR1(\x,\y) / den(\x,\y);
      R2(\x,\y) = numR2(\x,\y) / den(\x,\y);
    }
  ]
    \draw (0,0) node[below left] {$0$};
    \draw[->,>=latex] (-0.3,0) -- (1.3,0) node[right]{$x$};
    \draw[->,>=latex] (0,-0.3) -- (0,1.3) node[above]{$\tau(x,y)$};
    \draw[dashed] (0,1) node[left]{$1$} -- (1,1) -- (1,0) node[below]{$1$};
    
    \draw [domain=0:1, blue, line width=1, samples=50] plot (\x, 
      {(-\x + g(\x)) * (2*\x + 3*g(\x) + 1) / (2*\x*g(\x) - 5*g(\x) - 3)}
    );
   \draw[red,fill = red] ({1/2},{1/2}) circle (0.015);

    \begin{scope}[xshift=2cm,scale=3]
      \draw (0,0) node[below left] {$0$};
      \draw[->,>=latex] (-0.1,0) -- (0.4,0) node[right]{$R_1$};
      \draw[->,>=latex] (0,-0.1) -- (0,0.4) node[above]{$R_2$};
     \draw[dashed] (0,{1/6}) node[left]{$\frac{1}{6}$}-- ({1/6},{1/6}) -- ({1/6},0) node[below]{$\frac{1}{6}$};
     \draw[line width=1pt, blue]
     plot coordinates {
     (0.006622864, 0.006801155)
     (0.013161118, 0.013875608)
     (0.019620334, 0.021228980)
     (0.026007760, 0.028865308)
     (0.032332645, 0.036786375)
     (0.038606612, 0.044990834)
     (0.044844059, 0.053473075)
     (0.051062576, 0.062221782)
     (0.057283334, 0.071218098)
     (0.063531393, 0.080433354)
     (0.069835814, 0.089826270)
     (0.076229417, 0.099339644)
     (0.082747948, 0.108896545)
     (0.089428316, 0.118396190)
     (0.096305472, 0.127709903)
     (0.103407404, 0.136677836)
     (0.110747784, 0.145107569)
     (0.118315982, 0.152776166)
     (0.126064753, 0.159437594)
     (0.133896966, 0.164837336)
     (0.141654188, 0.168735110)
     (0.149111470, 0.170934488)
     (0.155983364, 0.171315042)
     (0.161944828, 0.169859444)
     (0.166666667, 0.166666667)
     (0.169859444, 0.161944828)
     (0.171315042, 0.155983364)
     (0.170934488, 0.149111470)
     (0.168735110, 0.141654188)
     (0.164837336, 0.133896966)
     (0.159437594, 0.126064753)
     (0.152776166, 0.118315982)
     (0.145107569, 0.110747784)
     (0.136677836, 0.103407404)
     (0.127709903, 0.096305472)
     (0.118396190, 0.089428316)
     (0.108896545, 0.082747948)
     (0.099339644, 0.076229417)
     (0.089826270, 0.069835814)
     (0.080433354, 0.063531393)
     (0.071218098, 0.057283334)
     (0.062221782, 0.051062576)
     (0.053473075, 0.044844059)
     (0.044990834, 0.038606612)
     (0.036786375, 0.032332645)
     (0.028865308, 0.026007760)
     (0.021228980, 0.019620334)
     (0.013875608, 0.013161118)
     (0.006801155, 0.006622864)
     (0.000000000, 0.000000000)
     };
     \draw[red,fill = red] ({1/6},{1/6}) circle (0.005);
    \end{scope}
  \end{tikzpicture}
  \caption{Numerical plots of $\tau(x,y)$ and $(R_1(x,y),R_2(x,y))$ along the arc of
 \eqref{eq:detOmega-affine} with $0<x<1$ containing the distinguished point $(\frac12,-\frac13)$ (right).
 Red points indicate values corresponding to $(x,y)=(\frac{1}{2}, -\frac13)$}
 \label{fig:tau}
\end{figure}
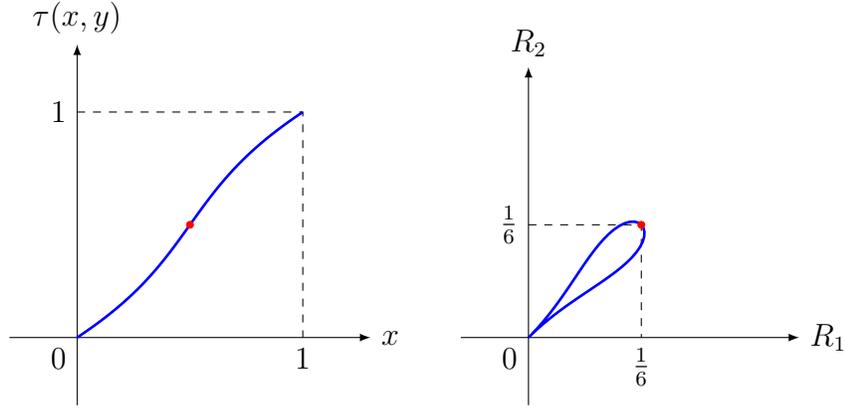

To determine the possible zeros and poles of these parameters
along the curve \(d(x,y)=0\),
we compute resultants with respect to \(y\).

For the parameter \(\tau\) we obtain
\[
\operatorname{Res}_y(d,N_\tau)
=
-384\,x(x-3)(x+2)(x-1)^3,
\]
\[
\operatorname{Res}_y(d,D_\tau)
=
48\,(x-3)(2x-1)(x-1)^2 .
\]

Hence the only possible zeros of \(\tau\) on the curve \(d(x,y)=0\)
occur at \(x=0\) or \(x=1\),
while the only possible pole in the interior of the interval
\(0<x<1\) occurs at \(x=\frac12\).

Similarly, the resultants for \(R_1\) and \(R_2\)
show that their possible zeros on the curve \(d(x,y)=0\)
occur only at the boundary points \(x=0\) and \(x=1\).

For instance,
\[
\operatorname{Res}_y(d,N_1)
=
98304\,x(x-3)^4(x-1)^7(3x^4-20x^3+7x^2+4x-3),
\]
and the remaining resultants are similar.

At the distinguished point$(x,y)=\left(\frac12,-\frac13\right)$,
a direct computation gives
\[
\tau=\frac12,\qquad
R_1=R_2=\frac16 .
\]

Therefore the intersection point lies strictly inside
the triangular disk bounded by \(\Delta\),
and the same holds for all other pairs by symmetry.

The strut intersects the interior of the triangular disk bounded by $\Delta$
if and only if
\[
0<\tau<1, \qquad
R_1>0, \qquad
R_2>0, \qquad
R_1+R_2<1.
\]

%
%
%
%

\subsection{Sign determination}

The elliptic curve contains the distinguished point
$
(x,y)=\left(\frac12,-\frac13\right),
$
which corresponds to the cuboctahedral configuration.
At this point we have
\[
\tau=\frac12,
\qquad
R_1=R_2=\frac16,
\qquad
1-R_1-R_2=\frac23.
\]

Hence all four quantities
\[
\tau,\qquad 1-\tau,\qquad R_1,\qquad R_2,\qquad 1-R_1-R_2
\]
are strictly positive at this point.

Since the resultant computations show that none of these functions
has a zero or pole on the open interval $0<x<1$,
their signs remain unchanged along the corresponding real arc
of the elliptic curve.

Therefore, for every realizable configuration with $0<x<1$,
we have
\[
0<\tau<1,
\qquad
R_1>0,
\qquad
R_2>0,
\qquad
R_1+R_2<1.
\]

Consequently the relevant strut intersects the interior
of the triangular disk bounded by $\Delta$.

\begin{thm}
For every realizable configuration with $0<x<1$,
the four triangular strut components of the tensegrity
are pairwise Hopf linked.
\end{thm}

\begin{proof}
At the cuboctahedral point the statement holds
by Theorem~2.
The analysis above shows that the intersection pattern
of the struts with the triangular disks
remains unchanged along the realizable family with $0<x<1$.
Hence the linking type cannot change.
Therefore the four triangular components remain
pairwise Hopf linked.
\end{proof}

\subsection{Arithmetic of the spectral curve}
For background on the arithmetic of elliptic curves, see \cite{Silverman}.
The condition that the stress matrix has a nontrivial kernel leads to the plane cubic
\begin{equation}
\label{eq:spectral_cubic}
x^2y + 3x^2 - xy - 3y^2 - 3x - 3y = 0 .
\end{equation}
This curve is nonsingular and therefore defines an elliptic curve over $\mathbb{Q}$.
Using SageMath, one finds that it is birational over $\mathbb{Q}$ to the elliptic curve
with Cremona label $30\mathrm{a}2$ in the LMFDB database~\cite{LMFDB}. 

A Weierstrass model of the curve is
\[
E:\quad Y^2 + 2XY + 72Y = X^3 + 48X^2 + 432X .
\]
This curve is isomorphic over $\mathbb{Q}$ to the short Weierstrass model
\begin{equation}
\label{eq:short_weierstrass}
E_0:\quad V^2 = U^3 - 384048\,U + 82988928 .
\end{equation}
The corresponding elliptic curve has Cremona label $30\mathrm{a}2$.
In particular, it has Mordell--Weil rank $0$ over $\mathbb{Q}$ and torsion subgroup
\[
E_0(\mathbb{Q})_{\mathrm{tors}} \cong \mathbb{Z}/2\mathbb{Z} \times \mathbb{Z}/6\mathbb{Z}.
\]

Explicitly, the birational transformation from the affine part of
\eqref{eq:spectral_cubic} to the Weierstrass model \eqref{eq:short_weierstrass} is given by
\[
U = \frac{1176x - 468y - 468}{2x - 3y - 3}, \qquad
V = \frac{15552(x + y - 1)}{2x - 3y - 3}.
\]
For example, the distinguished point $(x,y)=(\tfrac12,-\tfrac13)$ on
\eqref{eq:spectral_cubic} corresponds to the point
\[
(U,V)=(-276,12960)
\]
on the elliptic curve $E_0$.

\begin{rmk}
 The rational points on the curve $d(x,y)=0$
 are 
 $$(1,0),\ (3,3),\ \left(\frac12,-\frac13\right),\ (0,0),\ (-2,3),\ (3,-2),\ (1,-1),\ \left(\frac12,-\frac34\right),\ (-2,-2),$$ 
 together with the points at infinity
 $$(0:1:0), \quad (1:0:0), \quad (0:-1:1).$$
 Their x-coordinates are
 $$x = 0, \frac{1}{2}, 1, -2, 3.$$
 Remarkably, the resultant factorizations computed in Section 3.1 yield roots at precisely
 $$x = 0, 1, -2, 3,$$
 that is, at the x-coordinates of all rational points except 
 $x=\frac{1}{2}$. The exceptional value $x=\frac{1}{2}$
 corresponds to the cuboctahedral configuration — the unique physically realizable rational point.
 Thus the arithmetic of the elliptic curve appears to single out this configuration in a precise sense: it is the unique rational point whose x-coordinate does not appear as a root of any resultant. The mechanism behind this coincidence remains an open question.
\end{rmk}

%
%
%
%

\section{Trajectory of the strut--triangle intersection}
\label{sec:intersection}

In the configuration described above, each strut passes through
exactly one of the opposite triangular faces.
Let $(u,v)$ denote the barycentric coordinates of this intersection
point with respect to that triangle.

Using the parameters $(x,y)$ introduced in the previous section,
the intersection coordinates can be expressed as rational functions
\[
u = R_1(x,y), \qquad v = R_2(x,y).
\]

These two coordinates play symmetric roles,
since exchanging the two directions of the strut
interchanges the two coordinates.
Consequently the locus of $(u,v)$ is invariant under
\[
(u,v) \longmapsto (v,u).
\]

To exploit this symmetry we introduce the elementary symmetric
coordinates
\[
s = u+v, \qquad p = uv .
\]

Eliminating the parameters $(x,y)$ from the relations
\[
d(x,y)=0, \qquad
s = R_1(x,y)+R_2(x,y), \qquad
p = R_1(x,y)R_2(x,y),
\]
yields an algebraic relation
\[
G(s,p)=0.
\]
The polynomial $G(s,p)$ is explicitly presented in
Appendix $\ref{appendix:G}$.

Returning to the original coordinates we obtain a plane curve
\[
K(u,v)=G(u+v,uv)=0 .
\]

\begin{thm}
The locus of the strut--triangle intersection point $(u,v)$
is the symmetric septic curve
\[
K(u,v)=0 .
\]
This curve is invariant under the involution $(u,v)\mapsto(v,u)$.
It has singular points at
\[
(0,0), \qquad \left(\tfrac23,\tfrac23\right),
\]
and its geometric genus is equal to $1$.
\end{thm}

\begin{proof}[Sketch of proof]
The elimination can be carried out symbolically using computer algebra.
The resulting equation factors into several components, among which the
symmetric septic factor $K(u,v)$ corresponds to the actual geometric
locus.  A direct computation confirms that
\[
K(u,v)=K(v,u),
\]
and that the points $(0,0)$ and $(2/3,2/3)$ are singular.
The genus computation shows that the geometric genus
of the curve is $1$.
\end{proof}

Thus the trajectory curve is a singular plane model of an
elliptic curve.
In particular, both the configuration space
$d(x,y)=0$ and the trajectory curve $K(u,v)=0$
encode the same underlying elliptic geometry.

\appendix
\section{Two large factors in numerators of $R_1$ and $R_2$}

\begin{align*}
Q_1(x,y)=\;&
-4 x^{4} - 12 x^{3} y + x^{2} y^{2} + 9 x y^{3} + 8 x^{3}
 + 42 x^{2} y + 10 x y^{2} - 27 y^{3} \\
&+ 17 x^{2} - 15 x y - 38 y^{2} - 12 x - 15 y.
\end{align*}
\[
 Q_2(x,y)= 4 x^{4} - 4 x^{3} y - 13 x^{2} y^{2} + 9 x y^{3} - 8 x^{3} + 2 x^{2} y + 44 x y^{2} - 18 y^{3} + 7 x^{2} + 29 x y - 49 y^{2} + 6 x - 36 y - 9
\]

\section{The polynomial defining the curve of the strut-triangle intersection}
\label{appendix:G}
The polynomial $G(s,p)$ defined in section $\ref{sec:intersection}$ is given by
\begin{eqnarray*}
 G(s,p) & =&
 3 \mathit{s}^{7} - 7 \mathit{s}^{6} - 6 \mathit{s}^{5} \mathit{p} - \mathit{s}^{5} + 8 \mathit{s}^{4} \mathit{p} + 3 \mathit{s}^{3} \mathit{p}^{2} + 5 \mathit{s}^{4} + 28 \mathit{s}^{3} \mathit{p} - \mathit{s}^{2} \mathit{p}^{2} \\ &&- 3 \mathit{s}^{3} - 24 \mathit{s}^{2} \mathit{p} - 24 \mathit{s} \mathit{p}^{2} + \mathit{s}^{2} + 12 \mathit{s} \mathit{p} - 4 \mathit{p}.
\end{eqnarray*}

\section*{Acknowledgements}

The authors thank Risa Omura and Haruka Sakamoto for helpful discussions,
and for developing visualization tools
that assisted our exploration of the tensegrity configurations.


\end{document}